\documentclass{amsart}
\usepackage{amssymb,latexsym,amsfonts,amsthm,amsmath}

\theoremstyle{plain}
\newtheorem{theorem}{Theorem}

\begin{document}
\title{A short proof of Hara and Nakai's theorem}
\author{Byung-Geun Oh}\thanks{This work was supported by the research fund of Hanyang University(HY-2007-000-0000-4844)}

\address{Department of Mathematics Education, Hanyang University, 17 Haengdang-dong, Seongdong-gu, Seoul 133-791, Korea}
\email{bgoh@hanyang.ac.kr}
\date{\today}

\begin{abstract}
We give a short proof of the following theorem of Hara and Nakai: for a finitely bordered Riemann surface $R$, one can
find an upper bound of the corona constant of $R$ that depends only on the genus and the number
of boundary components of $R$.
\end{abstract}

\maketitle

\section*{The corona problem and Hara-Nakai's theorem}
For a given Riemann surface $R$, let $H^\infty (R)$ denote the
uniform algebra of bounded analytic functions on $R$. To avoid
pathological cases, we also assume that $H^\infty (R)$ separates
the points in $R$; i.e., for any $x_1, x_2 \in R$, $x_1 \ne x_2$, there
exists a function $f \in  H^\infty (R)$ with $f(x_1) \ne f(x_2)$.

We next consider the maximal ideal space $\mathcal{M}(R)$ of $H^\infty(R)$, and observe that
each $M \in \mathcal{M} (R)$ can be identified with $\varphi_M: H^\infty(R) \to \mathbb{C}$, where $\varphi_M$ is
the complex homomorphism that has $M$ as the kernel. This means that the maximal ideal space $\mathcal{M} (R)$ can be regarded
as a subspace of the dual space $(H^\infty(R))^*$ of $H^\infty(R)$. Moreover,
it also implies that we can equip $\mathcal{M} (R)$ with the Gelfand topology,
thus $\mathcal{M} (R)$ can be thought of a closed subspace of $(H^\infty(R))^*$  that
is contained in the unit sphere. For the details, see for example Chap.~V-1 of \cite{Gar}.

We know that each $\xi \in R$ corresponds to the maximal ideal
\[
 M_\xi = \{ f \in H^\infty (R) : f(\xi) = 0 \},
\]
hence $R$ can be naturally embedded into $\mathcal{M}(R)$ by the inclusion map $\iota : \xi \hookrightarrow M_\xi$.
Since we already provided $\mathcal{M}(R)$ with the Gelfand topology, one may ask the following:
``is $\iota (R)$ dense in $\mathcal{M} (R)$ with respect to the Gelfand topology?" This is a famous question that
is known as the \emph{corona problem}, and we will say that the \emph{corona theorem holds}
for $R$ if $\iota (R)$ dense in $\mathcal{M} (R)$. Otherwise
$R$ is said to have \emph{corona} ($= \mathcal{M}(R) \setminus \overline{\iota (R)}$).
Note that the complex homomorphism $\varphi_{M_\xi}$ associated with $M_\xi$ is nothing but the point evaluation map $\lambda_\xi : f \mapsto f(\xi)$.

It is known that the corona theorem holds for $R$ if and only if
the following function theoretical statement holds (cf. Chap.~4 of
\cite{Gam2}, Chap.~VIII of \cite{Gar}, or Chap.~12 of \cite{Du}):
for given $F_1, \ldots, F_n \in H^\infty (R)$ and $\delta \in
(0,1)$ such that
\begin{equation}\label{bound}
  \delta \leq \max_{1 \leq j \leq n} |F_j(\zeta)| \leq 1 \quad \mbox{for all } \zeta \in R,
\end{equation}
there exist $G_1, \ldots, G_n \in  H^\infty (R)$ that satisfy the
equation
\[
 F_1 G_1 + F_2 G_2 + \cdots + F_n G_n =1.
\]

We refer to $F_1, \ldots, F_n$ as \emph{corona data} of index $(n,
\delta)$ and $G_1, \ldots, G_n$ as \emph{corona solutions}
associated with the given corona data. The constant
\[
C(n,\delta,R) := \sup \inf \max \{ \|G_1\|_\infty, \ldots,
\| G_n \|_\infty \} 
\]
is called the ``corona constant" of $R$, where the supremum is
over all corona data satisfying \eqref{bound} and the infimum is
over all possible corona solutions associated with each corona
data. 

As usual, we interpret the infimum of an empty set as
infinity, thus if a Riemann surface $R$ has corona, the corona constant $C(n, \delta, R)$ must be infinite for some
index $(n, \delta)$. But what about the converse? If the corona theorem holds for $R$, is $C(n, \delta, R)$ finite for all 
indices $(n, \delta)$? The answer for this question is still unknown for general Riemann surfaces, 
but the answer is positive at least for finitely bordered Riemann surfaces, as the  
following theorem shows. 
 
\begin{theorem}[Hara and Nakai, \cite{HN}]\label{T}
For a given finitely bordered Riemann surface $R$, let $g(R)$ denote the genus of $R$ and 
$b(R)$ denote the number of boundary components of $R$. Then for each given index $(n, \delta)$ and numbers
$g \in \mathbb{N}\cup\{0\}$ and $b \in \mathbb{N}$, we have
\[
  \sup_{R \in \mathfrak{R}(g,b) } C(n, \delta, R) < \infty,
\]
where $\mathfrak{R}(g,b)$ is the collection of Riemann surfaces with $g(R) =g$ and $b(R) =b$.
\end{theorem}

The purpose of our paper is to give a short proof for this theorem. Note that Theorem~\ref{T} implies that one can find
an upper bound of the corona constant of a finitely bordered Riemann surface only depending on the index, genus of $R$ and
the number of boundary components of $R$.
The case $g=0$ was proved by Gamelin in \cite{Gam}, and
the case $g=0$ and $b=1$ is nothing but the famous Carleson's
corona theorem for the unit disc \cite{Ca}.

There are various planar domains and Riemann surfaces for which the
corona theorem holds (\cite{Ca}, \cite{Gam}, \cite{GJ}, \cite{JM},
\cite{St}, \cite{Al}, \cite{Be1}, \cite{Be2}, and more). On the other
hand, relatively a small number of Riemann surfaces are known to have corona. The first such example was constructed by Cole
(Chap.~4 of \cite{Gam2}), which was recently reconstructed in a
simpler way in \cite{Oh}, and other Riemann surfaces that have corona can be found in \cite{BD} and \cite{Ha}. 

The corona problem for general planar domains is still open, and
the answer is also unknown for a polydisc or a unit ball in
$\mathbb{C}^n$, $n \geq 2$.


\section*{Proof of Theorem~\ref{T}}
Our proof is based on the following three theorems and the
Carleson's corona theorem for the unit disc.

\begin{theorem}[Mitsuru Nakai, \cite{Na}]\label{TT}
Let $R$ and $R'$ be Riemann surfaces and $f: R' \to R$ an
$m$-sheeted branched covering map for some $m < \infty$. Then the
corona theorem holds for $R'$ if and only if it holds for $R$.
\end{theorem}

In Theorem~\ref{TT} Nakai considered only Riemann surfaces, that is, he considered
only \emph{connected} surfaces $R$ and $R'$. However, one
may check that the argument is still valid even when they are not
connected, i.e., in Theorem~\ref{TT} one can replace $R$ and $R'$ by disjoint unions of
Riemann surfaces.

Let $\mathbb{D}$ denote the unit disc. The next argument we will need for the proof of Theorem~\ref{T}
is the following.

\begin{theorem}[Ahlfors, \cite{Ahl}]\label{Ahl}
Suppose $R$ is a finitely bordered Riemann surface with $g(R) =g$ and $b(R) =b$.
Then there exists an $m$-sheeted branched covering map
$f:R \to \mathbb{D}$, called the \emph{Ahlfors map}, such that $b
\leq m \leq 2g + b$.
\end{theorem}

The last ingredient of our recipe is the following statement:

\begin{theorem}\label{TTT}
Let $\{ R_j \}$ be a sequence of Riemann surfaces. Then
\[
 \sup_{j} C(n, \delta, R_j) < \infty
 \]
for every index $(n, \delta)$ if and only if the corona theorem holds for $\bigsqcup_{j} R_j$, the disjoint union of $R_j$.
\end{theorem}

\begin{proof}
This theorem is essentially Lemma~3.1 of \cite{Gam}. In fact, in
\cite{Gam} the theorem was stated only for planar domains, but one can easily check that the
proof is valid for our case.
\end{proof}

Now we are ready to prove Theorem~\ref{T}.

\begin{proof}[Proof of Theorem~\ref{T}]
Suppose Theorem~\ref{T} is not true. Then there exist an index
$(n_0, \delta_0)$ and a sequence of finitely bordered Riemann surfaces
$\{ R_j \}$ with $g(R_j) = g $ and $b(R_j) = b$, $j=1,2,\ldots$,
such that
\begin{equation}\label{infty}
C(n_0, \delta_0, R_j) \to \infty
\end{equation}
as $j \to \infty$.
Furthermore by Theorem~\ref{Ahl}, we can find $m_j$-sheeted
branched covering maps $h_j : R_j \to D_j := \{ z: |z - 3j| < 1 \}$
with $b \leq m_j \leq 2g + b$ for all $j$.
However, by passing to a subsequence if necessary, we may assume that all the $m_j$'s are the same, that is, there 
exists a constant $m$ such that $m = m_j$ for all $j$.
Now let $D = \bigcup_j D_{j}$ and $\mathcal{R} = \bigsqcup_j R_{j}$.

Carleson's corona theorem for the unit disc \cite{Ca} (Theorem~\ref{T} for the case $g=0$ and $b=1$) implies that
$\sup_j C(n, \delta, D_{j} ) < \infty$ for any index $(n, \delta)$,
thus the corona theorem for $D$ follows from Theorem~\ref{TTT}. Then
by Nakai's theorem (Theorem~\ref{TT}), we see that the corona theorem also holds for $\mathcal{R}$,
because the map $h: \mathcal{R} \to D$ defined
by $h|_{R_{j}} = h_{j}$ is an $m$-sheeted branched covering.
According to Theorem~\ref{TTT}, however,  $\mathcal{R} = \bigsqcup_j R_{j}$ must have corona,
because $\sup_{j} C(n_0, \delta_0, R_j) = \infty$ by \eqref{infty}.
This contradiction completes the proof of Theorem~\ref{T}.
\end{proof}

We believe that what makes our proof significantly shorter than the proof of Hara and Nakai is the idea of applying
Hara's theorem (Theorem~\ref{TT}) to \emph{disjoint} sets, which is in fact due to Gamelin as Theorem~\ref{TTT} indicates.
The other parts of the proof is not very far from the original one
in the sense that both proofs use the Ahlfors maps and Nakai's theorem \cite{Na} (Theorems~\ref{Ahl} and \ref{TT} above).


\begin{thebibliography}{99}
\bibitem{Ahl} Lars L. Ahlfors, \emph{Open Riemann surfaces and extremal problems on compact
subregions}, Comment. Math. Helv. \textbf{24} (1950). 100--134.

\bibitem{Al} N. Alling, \emph{A proof of the corona conjecture for finite open Riemann surfaces},
Bull. Amer. Math. Soc. \emph{70} (1964), 110--112.

\bibitem{BD} D. E. Barrett and J. Diller, \emph{A new construction of Riemann surfaces with corona},
J. Geom. Anal. \textbf{8} (1998), 341--347.

\bibitem{Be1} M. Behrens, \emph{The corona conjecture for a class of infinitely connected domains},
Bull. Amer. Math. Soc. \textbf{76} (1970), 387--391.

\bibitem{Be2} M. Behrens, \emph{The maximal ideal space of algebras of bounded analytic functions on
infinitely connected domains}, Trans. Amer. Math. Soc.
\textbf{161} (1971), 359--379.

\bibitem{Ca} L. Carleson, \emph{Interpolations by bounded analytic functions and the corona problem},
Ann. of Math. (2) \textbf{76} (1962), 547--559.

\bibitem{Du} Peter L. Duren, \emph{Theory of $H\sp{p}$ spaces,} Pure and Applied
Mathematics, Vol. \textbf{38}, Academic Press, New York-London
1970.

\bibitem{Gam} T. W. Gamelin,  \emph{Localization of the corona problem},
Pacific J. Math. \textbf{34} (1970), 73--81.

\bibitem{Gam2} T. W. Gamelin, \emph{Uniform algebras and Jensen measures},
London Mathematical Society Lecture Note Series \textbf{32},
Cambridge University Press, Cambridge-New York, 1978.

\bibitem{Gar} J. B. Garnett, \emph{Bounded analytic functions},
Pure and Applied Mathematics \textbf{96}, Academic Press, Inc.,
New York-London, 1981.

\bibitem{GJ} J. B. Garnett and P. W. Jones, \emph{The Corona theorem for Denjoy domains},
Acta Math. \textbf{155} (1985), 27--40.

\bibitem{Ha} M. Hayashi, \emph{Bounded analytic functions on Riemann surfaces}, in the book:
Aspects of complex analysis, differential geometry, mathematical
physics and applications (St. Konstantin, 1998), World Sci.
Publishing, River Edge, NJ, 1999, 45--59.

\bibitem{HN} Masaru Hara and Mitsuru Nakai,
\emph{Corona theorem with bounds for finitely sheeted disks},
Tohoku Math. J. (2) \textbf{37} (1985), no. 2, 225--240.

\bibitem{JM} P. Jones and D. Marshall, \emph{Critical points of Green's functions,
harmonic measure and the corona problem}, Ark. Mat. \textbf{23}
(1985), 281--314.

\bibitem{Oh} Byung-Geun Oh, \emph{An explicit example of Riemann surfaces with large bounds on the corona
solutions},  Pacific J. Math. \textbf{228} (2006), no. 2, 297--304.

\bibitem{Na}  Mitsuru Nakai, \emph{The corona problem on finitely sheeted covering surfaces},
Nagoya Math. J. \textbf{92} (1983), 163--173.

\bibitem{St} E. L. Stout,
\emph{Bounded holomorphic functions on finite Riemann surfaces},
Trans. Amer. Math. Soc. \textbf{120} (1965), 255--285.
\end{thebibliography}
\end{document}